\documentclass[11pt,reqno]{amsart}

\usepackage[T1]{fontenc}
\usepackage[utf8]{inputenc}
\usepackage{lmodern}
\usepackage{microtype}
\usepackage{mathtools,amssymb}
\usepackage{enumitem}
\usepackage{xcolor}
\usepackage[colorlinks=true,linkcolor=blue,citecolor=blue,urlcolor=blue,filecolor=blue]{hyperref}
\hypersetup{
  pdftitle={Causal Invariance of the Contact Number of Pseudo-Riemannian Submanifolds},
  pdfauthor={Juan S. Gomez}
}
\numberwithin{equation}{section}

\allowdisplaybreaks
\setlist[enumerate]{label=\textup{(\roman*)},leftmargin=*,itemsep=2pt,topsep=4pt}
\setlist[itemize]{leftmargin=*,itemsep=2pt,topsep=4pt}

\theoremstyle{plain}
\newtheorem{theorem}{Theorem}[section]
\newtheorem{proposition}[theorem]{Proposition}
\newtheorem{lemma}[theorem]{Lemma}
\newtheorem{corollary}[theorem]{Corollary}

\theoremstyle{definition}
\newtheorem{definition}[theorem]{Definition}
\newtheorem{example}[theorem]{Example}

\theoremstyle{remark}
\newtheorem{remark}[theorem]{Remark}

\newcommand{\E}{\mathbb E}
\newcommand{\R}{\mathbb R}
\newcommand{\N}{\mathbb N}
\newcommand{\Span}{\operatorname{span}}
\newcommand{\cnum}{\mathfrak c}
\newcommand{\cplus}{\mathfrak c_{+}}
\newcommand{\cminus}{\mathfrak c_{-}}
\newcommand{\normal}{\perp}
\newcommand{\ip}[2]{\langle #1,#2\rangle}
\newcommand{\II}{\mathrm{II}}

\title[Causal invariance of the contact number]
{Causal Invariance of the Contact Number\\
of Pseudo-Riemannian Submanifolds}

\author{Juan S. G\'omez}
\address{Departamento de Matem\'aticas, Instituto E.S. Miguel Servet, 41020 Sevilla, Spain}
\email{jsalvador.gomez.edu@juntadeandalucia.es}

\date{}

\begin{document}
\raggedbottom

\begin{abstract}
Let $M_s^n$ be a non-degenerate pseudo-Riemannian submanifold of a
pseudo-Euclidean space with indefinite induced metric.  Restricting the
contact condition to spacelike or to timelike tangent directions gives two
a priori different one-sided contact numbers.  We prove that, for every
prescribed order, ordinary, spacelike and timelike contact are equivalent;
consequently, the three contact numbers coincide.  The proof is inductive.
A one-sided contact hypothesis first yields higher-order orthogonality
identities for the iterated covariant derivatives of the second fundamental
form.  Their diagonal scalar expressions are then extended from either unit
pseudo-sphere to the whole tangent space by a polynomial argument, and a
second induction proves causal invariance.  We also construct an explicit
Lorentzian surface in $\mathbb E_4^8$ whose ordinary, spacelike and timelike
contact numbers are all equal to six, and place the infinite-contact limit
in the semi-Riemannian theory of geodesic normal sections and helical
immersions.
\end{abstract}

\maketitle
\enlargethispage{2pt}

\section{Introduction}

Normal sections provide a direct comparison between the intrinsic and
extrinsic geometries of a submanifold.  Given a unit tangent vector $u$ at
a point $p$, one considers the affine space generated by $u$ and the
normal space at $p$.  The component of its intersection with the
submanifold through $p$ is the normal section in the direction $u$.  The
order to which this curve agrees with the geodesic having the same
initial data is the order of contact in that direction.

Submanifolds with geodesic normal sections were studied by Chen and
Verheyen \cite{ChenVerheyen}.  Chen and Li introduced the contact number
for Euclidean submanifolds in \cite{ChenLi}.  They proved that every
Euclidean submanifold has contact of order at least two, that contact of
order at least three is equivalent to isotropy, and that contact of order
at least four is equivalent to constant isotropy
\cite[Thm.~4.1]{ChenLi}.  They also proved that contact of order $k\geq3$ is equivalent to
requiring every unit tangent vector $u$ to be an eigenvector of the
shape operator $A_{(\overline\nabla^{\,j}\II)(u^{j+2})}$ for every
$0\leq j\leq k-3$; see \cite[Thm.~6.1]{ChenLi}.  Further
Euclidean developments include the study of surfaces with high contact
number \cite{ChenHighContact}, while Miura extended the invariant to
affine immersions \cite{Miura}.

Cabrerizo, Fern\'andez and G\'omez extended this theory to non-degenerate
pseudo-Riemannian submanifolds of pseudo-Euclidean space in \cite{CFG}.
The Euclidean low-order characterizations and this eigenvector criterion
became
\cite[Thms.~3.7 and~3.8]{CFG}.  In that paper, however, contact was
required simultaneously for every unit non-null vector.  Thus the unit
bundle used there was the full disjoint union
$UM_s^n=U_+M_s^n\cup U_-M_s^n$, and no separate spacelike or timelike contact number
was considered.  The broader theory of isotropic immersions originates
with O'Neill \cite{ONeillIsotropic}; relevant developments include
\cite{ItohOgiue,MaedaTsukada,BoumukiMaeda}.  In the indefinite setting,
subsequent work addressed rigidity \cite{CFGRigidity}, marginally trapped
surfaces \cite{CFGMarginallyTrapped}, and structural properties of
isotropic submanifolds \cite{CFGIsotropicSubmanifolds}.  We use throughout
the terms \emph{isotropic} and \emph{constant isotropic}.

The infinite-order endpoint of this circle of ideas is closely related to
the theory of helical immersions, rooted in Besse's study of strongly
harmonic manifolds \cite{Besse}.  In Riemannian space forms, helicality and
geodesic normal sections are equivalent; see Chen--Verheyen
\cite{ChenVerheyen}, Verheyen \cite{Verheyen}, and Hong--Houh
\cite{HongHouh}.  The indefinite theory is subtler.  Fueki introduced
pointed helical submanifolds in the pseudo-Riemannian setting \cite{Fueki},
and Miura developed the causal theory of helical geodesic immersions and
geodesic normal sections \cite{MiuraHelical,MiuraGNS}.  These results concern
the all-order situation; the question addressed here is finite and order by
order.

A question specific to indefinite signature is left open by the preceding
results.  Since the spacelike and timelike unit bundles are disjoint, one
may define contact by testing only $U_+M_s^n$, or only $U_-M_s^n$.  Denoting the
ordinary, spacelike and timelike contact numbers by $\cnum(M_s^n)$,
$\cplus(M_s^n)$ and $\cminus(M_s^n)$, respectively, the definitions give only
\begin{equation}\label{eq:intro-minimum}
 \cnum(M_s^n)=\min\{\cplus(M_s^n),\cminus(M_s^n)\}.
\end{equation}
There is no immediate geometric reason for the two one-sided numbers to
coincide: the relevant normal sections arise from distinct open regions
of the tangent bundle and may have different causal character.

The main result shows that this apparent distinction disappears.  For
every $k\geq2$, ordinary contact of order $k$, spacelike contact of order
$k$ and timelike contact of order $k$ are equivalent.  Consequently,
\begin{equation*}
 \cnum(M_s^n)=\cplus(M_s^n)=\cminus(M_s^n).
\end{equation*}
The proof is the original induction developed in the author's doctoral
dissertation.  Its principal ingredient is a one-sided family of
higher-order orthogonality identities.  These identities are obtained by
successively transferring covariant derivatives between the two factors,
with Codazzi and Ricci controlling the changes in the order of the
arguments.  Their diagonal terms are then propagated from either one of
the two unit pseudo-spheres to all tangent directions by a polynomial
argument.  A second induction, now on the order of contact, proves the
causal invariance.

The provenance of the ingredients is important.  The successive-derivative
comparison underlying the eigenvector criterion is due to Chen and Li in
the Euclidean case and appears in bilateral pseudo-Riemannian form in
\cite{CFG}.  Its fixed-sign formulation was recorded as
\cite[Thm.~4.12]{GomezThesis}; the same calculation is carried out while
keeping one causal sign throughout.  We therefore use that criterion as a
known result.  The finite-order ingredients that go beyond those earlier
criteria are the one-sided higher-order identities and the induction proving
causal invariance, recorded in the dissertation as Theorems~4.22 and~4.27.
Under the stronger, all-order hypothesis of geodesic normal sections, Miura
obtained related identities and causal equivalences \cite{MiuraGNS}.  The
present theorem is the finite-order counterpart: it assumes contact only
through a prescribed order $k$ and concludes the missing causal family at
that same order.

Besides the main theorem, we give an explicit Lorentzian surface in
$\mathbb E_4^8$ with contact number exactly six.  It is motivated by the
Euclidean contact-six construction of Chen and Li \cite[Example~6.8]{ChenLi}
and complements the Riemannian high-contact families studied later by Chen
\cite{ChenHighContact}.  The final discussion returns to the
infinite-contact limit, where expansions provide geodesic normal sections
but need not be helical.

Section~\ref{sec:preliminaries} fixes the notation and recalls the
one-sided form of the contact criterion.
Section~\ref{sec:higher-identities} develops the transfer identity and the
finite one-sided higher-order consequences.  The causal-invariance theorem
is proved in Section~\ref{sec:causal-invariance}.  The final section gives
low-order consequences, the Lorentzian contact-six example, and the relation
with geodesic normal sections and helical immersions.

\section{Preliminaries and contact orders}\label{sec:preliminaries}

We follow the standard semi-Riemannian conventions of \cite{ONeill}.  Let
$\E^{n+d}_{\nu}$ denote the $(n+d)$-dimensional pseudo-Euclidean space of
index $\nu$, and let
\begin{equation*}
 \psi\colon M_s^n\longrightarrow \E^{n+d}_{\nu}
\end{equation*}
be an isometric immersion of a connected pseudo-Riemannian manifold of
index $s$.  Throughout, $M_s^n$ is non-degenerate, $n\geq2$, and $0<s<n$.
We identify tangent vectors with their images under $d\psi$ and use the
same symbol for a curve in $M_s^n$ and its image under $\psi$.  The Levi-Civita connection of $M_s^n$, the normal connection and the ambient flat
connection are denoted by $\nabla$, $\nabla^\perp$ and
$\widetilde\nabla$, respectively.

The Gauss and Weingarten formulas are
\begin{align*}
 \widetilde\nabla_XY&=\nabla_XY+\II(X,Y),\\
 \widetilde\nabla_X\xi&=-A_\xi X+\nabla_X^\perp\xi,
\end{align*}
where $\II$ is the second fundamental form and $A_\xi$ is the shape operator
with respect to the normal field $\xi$.  They satisfy
\begin{equation}\label{eq:shape-II}
 \ip{A_\xi X}{Y}=\ip{\II(X,Y)}{\xi}.
\end{equation}

We use the curvature convention
\begin{equation*}
 R(X,Y)Z=\nabla_X\nabla_YZ-\nabla_Y\nabla_XZ-\nabla_{[X,Y]}Z,
\end{equation*}
and define $R^\normal$ analogously from $\nabla^\perp$.  Since the ambient space is
flat, the Gauss and Ricci equations are
\begin{align*}
 \ip{R(X,Y)Z}{W}
 &=\ip{\II(X,W)}{\II(Y,Z)}-\ip{\II(X,Z)}{\II(Y,W)},\\
 \ip{R^\normal(X,Y)\xi}{\eta}
 &=\ip{[A_\xi,A_\eta]X}{Y}.
\end{align*}

The covariant derivative of $\II$ is defined by
\begin{equation*}
 (\overline\nabla \II)(Y,Z,X)
 =\nabla_X^\perp(\II(Y,Z))-\II(\nabla_XY,Z)-\II(Y,\nabla_XZ).
\end{equation*}
The Codazzi equation takes the form
\begin{equation*}
 (\overline\nabla \II)(Y,Z,X)
 =(\overline\nabla \II)(X,Z,Y),
\end{equation*}
for all tangent vector fields $X,Y,Z$.  In particular,
$\overline\nabla \II$ is symmetric in all three arguments.  These standard formulas may be found, for instance, in \cite{ONeill}; in
the notation used here, compare also \cite[Eqs.~(1)--(6)]{CFG}.

Set $\overline\nabla^{\,0}\II=\II$.  For $j\geq1$, define recursively
\begin{equation*}
\begin{aligned}
 &(\overline\nabla^{\,j}\II)(X_1,\ldots,X_{j+2})\\
 &\quad=\nabla_{X_{j+2}}^\perp
 (\overline\nabla^{j-1}\II)(X_1,\ldots,X_{j+1})\\
 &\qquad-\sum_{a=1}^{j+1}
 (\overline\nabla^{j-1}\II)
 (X_1,\ldots,\nabla_{X_{j+2}}X_a,\ldots,X_{j+1}).
\end{aligned}
\end{equation*}
Thus the new differentiating argument occupies the last slot.  As usual,
$u^a$ denotes $a$ consecutive copies of $u$ among the arguments.  The
tensor $\overline\nabla^{\,j}\II$ remains symmetric in its first two
arguments.
For $a\geq3$, the Ricci commutation identity gives the particular formula
\begin{equation}\label{eq:Ricci-commutation}
\begin{aligned}
 &(\overline\nabla^{\,a-1}\II)(u^a,v)-(\overline\nabla^{\,a-1}\II)(u^{a-1},v,u)\\
 &\quad=R^\normal(v,u)(\overline\nabla^{\,a-3}\II)(u^{a-1})
 -\sum_{t=0}^{a-2}
 (\overline\nabla^{\,a-3}\II)(u^t,R(v,u)u,u^{a-2-t}).
\end{aligned}
\end{equation}
This is the only higher-order commutation formula needed below; compare
\cite[Eq.~(1.16)]{GomezThesis}.

Let $p\in M_s^n$ and let $u\in T_pM_s^n$ be unit and non-null, with
$\ip{u}{u}=\varepsilon\in\{1,-1\}$.  Denote by $\gamma_u$ the unit-speed
geodesic satisfying $\gamma_u(0)=p$ and $\gamma_u'(0)=u$.  The affine
normal-section space determined by $(p,u)$ is
\begin{equation*}
 E(p,u):=\psi(p)+\bigl(\R u\oplus T_p^\normal M_s^n\bigr).
\end{equation*}
Locally, the component of $\psi(M_s^n)\cap E(p,u)$ through $\psi(p)$ is a
regular non-degenerate curve.  Its unit-speed parametrization with initial
velocity $u$ is denoted by $\beta_u$ and called the normal section at
$(p,u)$.  Since $E(p,u)$ is affine,
\begin{equation}\label{eq:normal-section-derivatives}
 \beta_u^{(r)}(0)\in\R u\oplus T_p^\normal M_s^n,
 \qquad r\geq1.
\end{equation}

\begin{definition}
The geodesic $\gamma_u$ and the normal section $\beta_u$ have
\emph{contact of order $k$} at $(p,u)$ if
\begin{equation}\label{eq:directional-contact}
 \gamma_u^{(r)}(0)=\beta_u^{(r)}(0),
 \qquad 0\leq r\leq k,
\end{equation}
where the derivatives are taken in the ambient affine space.
\end{definition}

Set
\begin{equation*}
 U_+M_s^n=\{u\in TM_s^n:\ip{u}{u}=1\},
 \qquad
 U_-M_s^n=\{u\in TM_s^n:\ip{u}{u}=-1\}.
\end{equation*}
Both bundles are nonempty.

\begin{definition}
For $\varepsilon\in\{1,-1\}$, the immersion has
\emph{$\varepsilon$-contact of order $k$} if
\eqref{eq:directional-contact} holds for every $u\in U_\varepsilon M_s^n$.
We call this \emph{spacelike contact} when $\varepsilon=1$ and
\emph{timelike contact} when $\varepsilon=-1$.  The immersion has
\emph{contact of order $k$} if it has both spacelike and timelike contact
of order $k$.

The corresponding contact numbers are
\begin{align*}
 \cplus(M_s^n)&=\sup\{k\in\N : M_s^n\text{ has spacelike contact of order }k\},\\
 \cminus(M_s^n)&=\sup\{k\in\N : M_s^n\text{ has timelike contact of order }k\},\\
 \cnum(M_s^n)&=\sup\{k\in\N : M_s^n\text{ has contact of order }k\},
\end{align*}
with value $\infty$ when the relevant set is unbounded.
\end{definition}

By definition, \eqref{eq:intro-minimum} holds.  The following result is
\cite[Thm.~4.1(1)]{ChenLi} in the Euclidean case and
\cite[Thm.~3.7(1)]{CFG} in the pseudo-Euclidean setting.

\begin{proposition}\label{prop:contact-two}
Every pseudo-Riemannian submanifold $M_s^n$ of a pseudo-Euclidean space has contact
of order at least two.  Hence $\cnum(M_s^n)$, $\cplus(M_s^n)$ and $\cminus(M_s^n)$ are
at least two.
\end{proposition}

\begin{proof}
Let $u\in U_\varepsilon M_s^n$ and put $T=\beta_u'$.  Since $\beta_u$ has
constant speed, $\nabla_TT$ is orthogonal to $T$.  On the other hand,
\eqref{eq:normal-section-derivatives} and the Gauss formula imply
$\nabla_TT(0)\in\R u$.  Since $u$ is non-null,
$\nabla_TT(0)=0$.  Therefore
\begin{equation*}
 \beta_u''(0)=\II(u,u)=\gamma_u''(0),
\end{equation*}
and the position and velocity agree by construction.
\end{proof}

\begin{definition}
Following the pseudo-Riemannian terminology of \cite{CFG}, extending
O'Neill's notion \cite{ONeillIsotropic}, the immersion is
\emph{isotropic} at $p$ if there exists $\lambda(p)\in\R$ such that
\begin{equation}\label{eq:isotropy}
 \ip{\II(u,u)}{\II(u,u)}
 =\lambda(p)\ip{u}{u}^2
 \qquad\text{for every }u\in T_pM_s^n.
\end{equation}
It is \emph{isotropic} if this holds at every point.  The resulting
function $\lambda$ is called the \emph{isotropy function}.  If $\lambda$
is constant, the immersion is said to be \emph{constant isotropic}, and
its value is called the \emph{isotropy constant}.
\end{definition}

The low-order relation with isotropy is already known and will be used
in the initial step of the main induction.

\begin{theorem}[Low-order contact]
\label{thm:known-low-orders}
A pseudo-Riemannian submanifold $M_s^n$ of a pseudo-Euclidean space has
contact of order at least three if and only if it is isotropic, and
contact of order at least four if and only if it is constant isotropic.
\end{theorem}

Theorem~\ref{thm:known-low-orders} is
\cite[Thm.~3.7(2)--(3)]{CFG}, extending
\cite[Thm.~4.1(2)--(3)]{ChenLi}. It also covers the indefinite
two-dimensional case, so no additional dimensional hypothesis is
required.

We shall also use the following eigenvector characterization of
one-sided contact.

\begin{theorem}[One-sided eigenvector criterion]
\label{thm:contact-criterion}
Let $k\geq3$ and fix $\varepsilon\in\{1,-1\}$. The immersion has
$\varepsilon$-contact of order $k$ if and only if, for every
$u\in U_\varepsilon M_s^n$ and every $0\leq j\leq k-3$, the vector $u$
is an eigenvector of the shape operator
\[
A_{(\overline\nabla^{\,j}\II)(u^{j+2})}.
\]
Consequently, the immersion has contact of order $k$ if and only if
the same eigenvector condition holds for every unit non-null tangent
vector.
\end{theorem}

The one-sided formulation is Theorem~4.12 of
\cite{GomezThesis}. As noted there, its proof follows the
successive-derivative argument of Chen and Li
\cite[Lemmas~6.2 and~6.3]{ChenLi}, with the causal sign fixed
throughout. The corresponding bilateral pseudo-Riemannian criterion is
\cite[Thm.~3.8]{CFG}. We shall therefore use
Theorem~\ref{thm:contact-criterion} without reproducing that calculation.

\begin{remark}
The eigenvector condition in Theorem~\ref{thm:contact-criterion} is not a
new replacement for the Chen--Li criterion.  Its role here is different:
the same condition is isolated on one causal unit bundle, so that the
question becomes whether information on $U_+M_s^n$ determines $U_-M_s^n$, and
conversely.
\end{remark}

\section{One-sided higher-order identities}\label{sec:higher-identities}

We first isolate the transfer calculation underlying the original proof.
For $r\geq2$ and $\varepsilon\in\{1,-1\}$, let
$\mathcal P_r^\varepsilon$ denote the following assertion: whenever
$i,j\geq0$, $i+j\leq r-3$, $0\leq a\leq j+1$, and $u,v$ are orthonormal
at the same point with $\langle u,u\rangle=\varepsilon$, one has
\begin{equation*}
 \langle(\overline\nabla^{\,i}\II)(u^{i+2}),
 (\overline\nabla^{\,j}\II)(u^a,v,u^{j-a+1})\rangle=0.
\end{equation*}
Thus $\mathcal P_2^\varepsilon$ is vacuous.  The terminal position
$a=j+1$ is included because it occurs when a diagonal expression is
differentiated.

\begin{lemma}[Transfer identity]\label{lem:transfer-identity}
Fix $N\geq0$ and $\varepsilon\in\{1,-1\}$, and assume that
$\mathcal P_{N+2}^\varepsilon$ holds.  If $i+j=N$, $0\leq a\leq j+1$,
and $u,v$ are orthonormal with
$\langle u,u\rangle=\varepsilon$, then
\begin{equation}\label{eq:transfer-identity}
\begin{aligned}
 &\langle(\overline\nabla^{\,i}\II)(u^{i+2}),
 (\overline\nabla^{\,j}\II)(u^a,v,u^{j-a+1})\rangle\\
 &\qquad=(-1)^j
 \langle(\overline\nabla^{\,N}\II)(u^{N+2}),\II(u,v)\rangle.
\end{aligned}
\end{equation}
\end{lemma}

\begin{proof}
For $N=0$, formula \eqref{eq:transfer-identity} is just the symmetry of
$\II$.  Assume $N\geq1$ and set
\begin{equation*}
 B_{i,j,a}:=
 \langle(\overline\nabla^{\,i}\II)(u^{i+2}),
 (\overline\nabla^{\,j}\II)(u^a,v,u^{j-a+1})\rangle.
\end{equation*}
Choose local orthonormal extensions $X,Y$ of $u,v$ such that
$\nabla_ZX=\nabla_ZY=0$ at $p$ for every $Z\in T_pM_s^n$.
If $0<a\leq j$, differentiating in the $X$-direction gives
\begin{equation}\label{eq:transfer-step}
\begin{aligned}
 B_{i,j,a}
 &=X\!\left(
 \langle(\overline\nabla^{\,i}\II)(X^{i+2}),
 (\overline\nabla^{\,j-1}\II)(X^a,Y,X^{j-a})\rangle\right)_p\\
 &\quad-
 \langle(\overline\nabla^{\,i+1}\II)(u^{i+3}),
 (\overline\nabla^{\,j-1}\II)(u^a,v,u^{j-a})\rangle.
\end{aligned}
\end{equation}
The differentiated function vanishes by
$\mathcal P_{N+2}^\varepsilon$, since the total derivative order in it is
$N-1$.  Iterating \eqref{eq:transfer-step} therefore yields
\begin{equation}\label{eq:transfer-intermediate}
 B_{i,j,a}=(-1)^{j-a+1}
 \langle
 (\overline\nabla^{\,N-a+1}\II)(u^{N-a+3}),
 (\overline\nabla^{\,a-1}\II)(u^a,v)\rangle.
\end{equation}

We next prove, for $1\leq\ell\leq N+1$ and
$I+\ell-1=N$, that
\begin{equation}\label{eq:terminal-transfer}
 \langle(\overline\nabla^{\,I}\II)(u^{I+2}),
 (\overline\nabla^{\,\ell-1}\II)(u^\ell,v)\rangle
 =(-1)^{\ell-1}
 \langle(\overline\nabla^{\,N}\II)(u^{N+2}),\II(u,v)\rangle.
\end{equation}
For $\ell=1$ this is immediate.  For $\ell=2$, Codazzi permits
the last argument $v$ to be interchanged with one of the copies of $u$.
Differentiating the lower-order identity
$\langle(\overline\nabla^{\,I}\II)(X^{I+2}),\II(X,Y)\rangle=0$ in the
$X$-direction then gives \eqref{eq:terminal-transfer} with the required
minus sign.

Let $\ell\geq3$ and assume \eqref{eq:terminal-transfer} has been proved
with $\ell-1$ in place of $\ell$.  The Ricci commutation identity
\eqref{eq:Ricci-commutation} compares
$(\overline\nabla^{\,\ell-1}\II)(u^\ell,v)$ with
$(\overline\nabla^{\,\ell-1}\II)(u^{\ell-1},v,u)$.  After pairing with
$(\overline\nabla^{\,I}\II)(u^{I+2})$, all tangent-curvature terms vanish by
$\mathcal P_{N+2}^\varepsilon$, because their total derivative order is
$N-2$ and the curvature symmetries give $R(v,u)u\perp u$.

The normal-curvature term vanishes as well.  Put
$\xi=(\overline\nabla^{\,\ell-3}\II)(u^{\ell-1})$ and
$\eta=(\overline\nabla^{\,I}\II)(u^{I+2})$.  The lower-order identities with second factor
$\II(u,w)$ show that $A_\xi u=\alpha u$ and $A_\eta u=\beta u$ for some
real numbers $\alpha$ and $\beta$.  Since the shape operators are
self-adjoint, the Ricci equation gives
\begin{equation*}
\begin{aligned}
 \langle R^\perp(v,u)\xi,\eta\rangle
 &=\langle[A_\xi,A_\eta]v,u\rangle\\
 &=\alpha\langle A_\eta v,u\rangle
   -\beta\langle A_\xi v,u\rangle=0.
\end{aligned}
\end{equation*}
Consequently,
\begin{equation*}
\begin{aligned}
 &\langle(\overline\nabla^{\,I}\II)(u^{I+2}),
 (\overline\nabla^{\,\ell-1}\II)(u^\ell,v)\rangle\\
 &\quad=
 \langle(\overline\nabla^{\,I}\II)(u^{I+2}),
 (\overline\nabla^{\,\ell-1}\II)(u^{\ell-1},v,u)\rangle\\
 &\quad=-
 \langle(\overline\nabla^{\,I+1}\II)(u^{I+3}),
 (\overline\nabla^{\,\ell-2}\II)(u^{\ell-1},v)\rangle.
\end{aligned}
\end{equation*}
Here the last equality is obtained by differentiating, in the
$X$-direction, the lower-order identity
$\langle(\overline\nabla^{\,I}\II)(X^{I+2}),
(\overline\nabla^{\,\ell-2}\II)(X^{\ell-1},Y)\rangle=0$.
The induction hypothesis applied to the last scalar product proves
\eqref{eq:terminal-transfer}.

For $1\leq a\leq j$, combining
\eqref{eq:transfer-intermediate} and
\eqref{eq:terminal-transfer} gives
\eqref{eq:transfer-identity}.  The case $a=0$ reduces to $a=1$ by the
symmetry of the first two arguments of $\overline\nabla^{\,j}\II$, while the case
$a=j+1$ is exactly \eqref{eq:terminal-transfer} with
$\ell=j+1$ and $I=i$.  This completes the proof.
\end{proof}

\begin{theorem}[One-sided higher-order identities]
\label{thm:higher-identities}
Let $k\geq3$ and fix $\varepsilon\in\{1,-1\}$.  Assume that $M_s^n$ has
$\varepsilon$-contact of order $k$.  Then the following assertions hold.
\begin{enumerate}
 \item If $i,j\geq0$, $i+j\leq k-3$, $0\leq a\leq j+1$, and $u,v$ are
 orthonormal with $\langle u,u\rangle=\varepsilon$, then
 \begin{equation*}
  \langle(\overline\nabla^{\,i}\II)(u^{i+2}),
  (\overline\nabla^{\,j}\II)(u^a,v,u^{j-a+1})\rangle=0.
 \end{equation*}
 \item If $i,j\geq0$ and $i+j\leq k-3$, there exists a smooth function
 $\lambda_{ij}$ on $M_s^n$ such that, for every $w\in TM_s^n$,
 \begin{equation}\label{eq:scalar-polynomial}
 \langle(\overline\nabla^{\,i}\II)(w^{i+2}),(\overline\nabla^{\,j}\II)(w^{j+2})\rangle
 =
 \begin{cases}
  \lambda_{ij}\langle w,w\rangle^{(i+j+4)/2},&i+j\text{ even},\\
  0,&i+j\text{ odd}.
 \end{cases}
 \end{equation}
 \item If $k\geq4$ and $i+j\leq k-4$, then $\lambda_{ij}$ is constant.
\end{enumerate}
\end{theorem}

\begin{proof}
We first prove \textup{(i)} by the induction used in the dissertation.
The assertion $\mathcal P_2^\varepsilon$ is vacuous.  Suppose that
$3\leq r\leq k$ and that $\mathcal P_{r-1}^\varepsilon$ has been
established.  Put $N=r-3$.  For $i+j=N$, the transfer identity gives
\begin{equation*}
 \langle(\overline\nabla^{\,i}\II)(u^{i+2}),
 (\overline\nabla^{\,j}\II)(u^a,v,u^{j-a+1})\rangle
 =(-1)^j\langle(\overline\nabla^{\,N}\II)(u^{N+2}),\II(u,v)\rangle.
\end{equation*}
The right-hand side is
$(-1)^j\langle A_{(\overline\nabla^{\,N}\II)(u^{N+2})}u,v\rangle$ and therefore vanishes by
Theorem~\ref{thm:contact-criterion}.  Thus
$\mathcal P_r^\varepsilon$ holds.  Induction gives
$\mathcal P_k^\varepsilon$, which is precisely \textup{(i)}.

Fix $p\in M_s^n$ and $i+j\leq k-3$, and define the homogeneous polynomial
\begin{equation*}
 F_{ij}(w)=
 \langle(\overline\nabla^{\,i}\II)(w^{i+2}),(\overline\nabla^{\,j}\II)(w^{j+2})\rangle,
 \qquad w\in T_pM_s^n.
\end{equation*}
Let
\begin{equation*}
 \Sigma_p^\varepsilon=\{w\in T_pM_s^n:\langle w,w\rangle=\varepsilon\}.
\end{equation*}
If $u\in\Sigma_p^\varepsilon$ and $v\in u^\perp$, differentiation gives
\begin{equation}\label{eq:dFij}
\begin{aligned}
 d(F_{ij})_u(v)
 &=\sum_{a=0}^{i+1}
 \langle(\overline\nabla^{\,i}\II)(u^a,v,u^{i-a+1}),(\overline\nabla^{\,j}\II)(u^{j+2})\rangle\\
 &\quad+\sum_{a=0}^{j+1}
 \langle(\overline\nabla^{\,i}\II)(u^{i+2}),(\overline\nabla^{\,j}\II)(u^a,v,u^{j-a+1})\rangle=0.
\end{aligned}
\end{equation}
For unit non-null $v$ this follows from \textup{(i)}, after interchanging
the two factors in the first sum; linearity and continuity give it for
every $v\in u^\perp$.  Hence $F_{ij}$ is constant on each connected
component $A$ of $\Sigma_p^\varepsilon$.

Assume first that $i+j$ is even and put
$m=(i+j+4)/2$.  Let $C$ be the value of $F_{ij}$ on $A$ and set
$\lambda_{ij}(p)=C\varepsilon^m$.  By homogeneity,
\begin{equation}\label{eq:even-cone}
 F_{ij}(w)=\lambda_{ij}(p)\langle w,w\rangle^m
\end{equation}
whenever the normalization of $w$ lies in $A$.  Fix one such vector $u$
and take an arbitrary $z\in T_pM_s^n$.  For sufficiently small $t$,
$u_t=u+tz$ remains in the same open cone, and
\eqref{eq:even-cone} holds for $u_t$.  Both sides are polynomials in $t$
of degree at most $i+j+4$; equality of their leading coefficients gives
\eqref{eq:even-cone} with $z$ in place of $u$.  Since $z$ is arbitrary,
the even case of \eqref{eq:scalar-polynomial} holds on all of $T_pM_s^n$.

Suppose now that $i+j$ is odd and set $d=i+j+4$, which is odd.  On the
open cone over $A$, homogeneity gives
\begin{equation}\label{eq:odd-cone}
 F_{ij}(w)=C\bigl(\varepsilon\langle w,w\rangle\bigr)^{d/2}.
\end{equation}
Choose $u\in A$ and a unit non-null vector $z\perp u$.  Then
$\operatorname{span}\{u,z\}$ is non-degenerate.  After substituting
$u_t=u+tz$ in \eqref{eq:odd-cone} and squaring, one obtains, for small
$t$,
\begin{equation}\label{eq:odd-square}
 P(t)^2=C^2Q(t)^d,
 \qquad
 P(t)=F_{ij}(u_t),
 \quad
 Q(t)=\varepsilon\langle u_t,u_t\rangle.
\end{equation}
The quadratic polynomial $Q$ has two simple roots over $\mathbb C$.
If $C\neq0$, every root of $Q$ would occur in the right-hand side of
\eqref{eq:odd-square} with the odd multiplicity $d$, whereas every root
of the square $P^2$ has even multiplicity.  Thus $C=0$.  The polynomial $F_{ij}$ vanishes on a
nonempty open cone and hence vanishes identically.  This proves
\textup{(ii)}; in the odd case we take $\lambda_{ij}=0$.  Smoothness in
the even case follows by evaluating \eqref{eq:scalar-polynomial} on a
local unit field.

Finally, let $k\geq4$, let $i+j\leq k-4$, and suppose $i+j$ is even.
Put $m=(i+j+4)/2$.  Fix $p\in M_s^n$ and let $z\in T_pM_s^n$ range over the
nonempty open cone with causal sign opposite to $\varepsilon$.  Since the
metric is indefinite, the non-degenerate space $z^\perp$ contains a unit
vector $u$ satisfying
$\langle u,u\rangle=\varepsilon$.  Choose a local unit field $X$ with
$X_p=u$ and $\nabla_zX=0$ at $p$.  Differentiating
\eqref{eq:scalar-polynomial} along $z$ gives
\begin{equation*}
\begin{aligned}
 \varepsilon^m z(\lambda_{ij})
 &=\langle(\overline\nabla^{\,i+1}\II)(u^{i+2},z),(\overline\nabla^{\,j}\II)(u^{j+2})\rangle\\
 &\quad+\langle(\overline\nabla^{\,i}\II)(u^{i+2}),(\overline\nabla^{\,j+1}\II)(u^{j+2},z)\rangle=0.
\end{aligned}
\end{equation*}
Both terms vanish by \textup{(i)}, because
$i+j+1\leq k-3$; the statement for a non-unit $z$ follows by linearity.
Thus the linear form $d\lambda_{ij}|_p$ vanishes on a nonempty open cone,
and therefore vanishes identically.  Since $p$ is arbitrary and $M_s^n$ is
connected, $\lambda_{ij}$ is constant.  When $i+j$ is odd there is
nothing to prove.  This establishes \textup{(iii)}.
\end{proof}

\begin{remark}
Lemma~\ref{lem:transfer-identity} isolates the calculation represented
by formulas~(4.37)--(4.40) of \cite{GomezThesis}, and
Theorem~\ref{thm:higher-identities} is a sharpened formulation of
\cite[Thm.~4.22]{GomezThesis}.  The terminal position $a=j+1$, needed in
\eqref{eq:dFij}, is made explicit here and follows from the same Ricci
commutation used in the original proof.  Apart from this clarification,
the transfer induction and the even--odd polynomial argument are
unchanged.
\end{remark}

\begin{remark}[Finite order versus geodesic normal sections]
Under the stronger hypothesis of geodesic normal sections, Miura proved
all-order counterparts of the orthogonality, parity and constancy properties
above; see \cite[Lemmas~3.4--3.6]{MiuraGNS}.  His
\cite[Lemma~3.3]{MiuraGNS} also relates the spacelike, timelike and full
geodesic-normal-section conditions.  Theorem~\ref{thm:higher-identities} is
instead a finite-order statement: only contact through order $k$ is assumed,
and only the tensors with total derivative order at most $k-3$ are
controlled.
\end{remark}

\section{Causal invariance}\label{sec:causal-invariance}

We now carry out the original induction on the contact order.

\begin{theorem}[Causal invariance of contact]\label{thm:causal-invariance}
Let $\psi\colon M_s^n\to\E^{n+d}_\nu$ be an isometric immersion
with $0<s<n$, and let $k\geq2$.  The following assertions are equivalent:
\begin{enumerate}
 \item $M_s^n$ has contact of order $k$;
 \item $M_s^n$ has spacelike contact of order $k$;
 \item $M_s^n$ has timelike contact of order $k$.
\end{enumerate}
In particular,
\begin{equation}\label{eq:causal-invariance}
 \cnum(M_s^n)=\cplus(M_s^n)=\cminus(M_s^n).
\end{equation}
\end{theorem}

\begin{proof}
Full contact plainly implies each one-sided condition.  We prove that
spacelike contact of order $k$ implies full contact; the timelike proof is
identical.

The case $k=2$ is Proposition~\ref{prop:contact-two}.  If $k=3$,
Theorem~\ref{thm:higher-identities}\textup{(ii)}, with $i=j=0$, shows
that \eqref{eq:isotropy} holds at every point, with
$\lambda=\lambda_{00}$.  Thus $M_s^n$ is isotropic, and
Theorem~\ref{thm:known-low-orders} yields full contact of order three.

Let $k>3$ and assume that the conclusion is true for order $k-1$.
Suppose that $M_s^n$ has spacelike contact of order $k$.  It then has
spacelike contact of order $k-1$, and the induction hypothesis gives full
contact of order $k-1$.

Fix $p\in M_s^n$.  For $0\leq j\leq k-3$, define on the unit non-null vectors
at $p$
\begin{equation*}
 f_j(w)=\langle(\overline\nabla^{\,j}\II)(w^{j+2}),\II(w,w)\rangle.
\end{equation*}
By Theorem~\ref{thm:higher-identities}\textup{(ii)}, applied to the
spacelike contact of order $k$, the restriction of $f_j$ to either unit
pseudo-sphere is constant.  Let $u,v\in T_pM_s^n$ be orthonormal.  Since
$v\in T_u\Sigma_p^{\langle u,u\rangle}$, differentiation gives
\begin{equation}\label{eq:main-derivative}
 0=\sum_{a=0}^{j+1}
 \langle(\overline\nabla^{\,j}\II)(u^a,v,u^{j-a+1}),\II(u,u)\rangle
 +2\langle(\overline\nabla^{\,j}\II)(u^{j+2}),\II(u,v)\rangle.
\end{equation}

The full contact of order $k-1$ and
Theorem~\ref{thm:higher-identities}\textup{(i)} provide all
orthogonality identities of total derivative order at most $j-1$.
Therefore Lemma~\ref{lem:transfer-identity}, used at the borderline total
order $N=j$ with $i=0$, gives, for every $a$,
\begin{equation*}
 \langle(\overline\nabla^{\,j}\II)(u^a,v,u^{j-a+1}),\II(u,u)\rangle
 =(-1)^j\langle(\overline\nabla^{\,j}\II)(u^{j+2}),\II(u,v)\rangle.
\end{equation*}
There are $j+2$ terms in the sum in
\eqref{eq:main-derivative}.  Hence
\begin{equation*}
 \bigl((j+2)(-1)^j+2\bigr)
 \langle(\overline\nabla^{\,j}\II)(u^{j+2}),\II(u,v)\rangle=0.
\end{equation*}
The coefficient equals $j+4$ when $j$ is even and $-j$ when $j$ is odd,
so it never vanishes.  Consequently,
\begin{equation*}
 \langle A_{(\overline\nabla^{\,j}\II)(u^{j+2})}u,v\rangle
 =\langle(\overline\nabla^{\,j}\II)(u^{j+2}),\II(u,v)\rangle=0
\end{equation*}
for every unit non-null $v\perp u$.  By linearity and continuity the same
holds for every $v\in u^\perp$.  Thus $u$ is an eigenvector of
$A_{(\overline\nabla^{\,j}\II)(u^{j+2})}$.  Since $u$ was an arbitrary
unit non-null vector and $0\leq j\leq k-3$,
Theorem~\ref{thm:contact-criterion} gives full contact of order $k$.
This completes the induction.  Equality \eqref{eq:causal-invariance}
follows directly from the equivalence for every $k$.
\end{proof}

\begin{remark}[Proof perspective]
The preceding proof is the inductive argument originally developed in
\cite[Thm.~4.27]{GomezThesis}.  Its conceptual content is visible in
Theorem~\ref{thm:higher-identities}: the diagonal scalar expressions are
homogeneous polynomials in the tangent direction.  Once the induction has
produced the relevant identity on either the spacelike or the timelike
open cone, polynomial continuation determines it on the whole tangent
space.  The polynomial argument does not replace the transfer induction;
it explains why the resulting higher-order obstruction cannot distinguish
the two causal families.
\end{remark}

\begin{remark}
To determine the contact number it is enough to study only spacelike
normal sections or only timelike normal sections.  This behavior is not
typical of pseudo-Riemannian geometry: there are Lorentzian manifolds for
which all spacelike geodesics are complete while some timelike geodesics
are incomplete; see \cite[p.~154]{ONeill}.  In this precise sense the
contact number is causally invariant.
\end{remark}

\section{Consequences, examples, and the infinite-contact limit}

Combining Theorem~\ref{thm:causal-invariance} with the previously known
ordinary-contact characterizations gives the following immediate
consequences.

\begin{corollary}\label{cor:contact-three}
The following conditions are equivalent:
\begin{enumerate}
 \item $M_s^n$ has contact of order at least three;
 \item $M_s^n$ has spacelike contact of order at least three;
 \item $M_s^n$ has timelike contact of order at least three;
 \item $M_s^n$ is isotropic.
\end{enumerate}
\end{corollary}

\begin{proof}
Apply Theorems~\ref{thm:causal-invariance} and~\ref{thm:known-low-orders}.
\end{proof}

\begin{corollary}\label{cor:contact-four}
The following conditions are equivalent:
\begin{enumerate}
 \item $M_s^n$ has contact of order at least four;
 \item $M_s^n$ has spacelike contact of order at least four;
 \item $M_s^n$ has timelike contact of order at least four;
 \item $M_s^n$ is constant isotropic.
\end{enumerate}
\end{corollary}

\begin{proof}
Again, this follows from Theorems~\ref{thm:causal-invariance}
and~\ref{thm:known-low-orders}.
\end{proof}

For ordinary contact, the equivalences between the first and fourth
conditions in Corollaries~\ref{cor:contact-three} and
\ref{cor:contact-four} are precisely \cite[Thm.~3.7(2)--(3)]{CFG},
extending \cite[Thm.~4.1(2)--(3)]{ChenLi}.  The new content is that either
the spacelike or the timelike unit bundle alone detects the same property.

The next two standard consequences are
\cite[Corollaries~3.9 and~3.11]{CFG}.

\begin{corollary}
Each of the following conditions implies $\cnum(M_s^n)=\infty$:
\begin{enumerate}
 \item $M_s^n$ is isotropic and $\overline\nabla \II=0$;
 \item $M_s^n$ is totally umbilical.
\end{enumerate}
In either case, Theorem~\ref{thm:causal-invariance} also gives
$\cplus(M_s^n)=\cminus(M_s^n)=\infty$.
\end{corollary}

\begin{proof}
Under the first assumption, the criterion for $j=0$ follows from
third-order contact, while $\overline\nabla^{\,j}\II=0$ for every $j\geq1$.
Theorem~\ref{thm:contact-criterion} applies for every $k$.

If $M_s^n$ is totally umbilical, then
$\II(X,Y)=\ip{X}{Y}\mathbf H$.  Substitution into Codazzi gives
\begin{equation*}
 \ip{Y}{Z}\nabla_X^\perp\mathbf H=\ip{X}{Z}\nabla_Y^\perp\mathbf H.
\end{equation*}
For a non-null $X$, choose a non-null $Y\perp X$ and set $Z=Y$; then
$\nabla_X^\perp\mathbf H=0$.  By continuity,
$\nabla^\perp\mathbf H=0$, so $\II$ is parallel.  Total umbilicity also
implies isotropy, and the first assertion applies.
\end{proof}

\begin{example}[A Lorentzian cylinder]\label{ex:cylinder}
This is the example recorded in \cite[Remark~3.10]{CFG}.  In
Lorentz--Minkowski three-space $\mathbb L^3$, with metric
$-dx^2+dy^2+dz^2$, consider
\begin{equation*}
 M_1^2=\R\times\mathbb S^1
 =\{(x,y,z)\in\mathbb L^3:y^2+z^2=1\}.
\end{equation*}
Its second fundamental form is parallel.  It vanishes on the unit timelike
direction tangent to the $x$-factor and is nonzero on the unit spacelike
direction tangent to the circle.  Hence the immersion is not isotropic.
Corollary~\ref{cor:contact-three}, together with
Proposition~\ref{prop:contact-two} and
Theorem~\ref{thm:causal-invariance}, gives
\begin{equation*}
 \cnum(M_1^2)=\cplus(M_1^2)=\cminus(M_1^2)=2.
\end{equation*}
\end{example}

\begin{remark}[The normal index in finite high-contact examples]
There is an elementary but decisive restriction.  If an isotropic
immersion $M_s^n\to\mathbb E_\nu^{n+d}$ has indefinite induced metric and
$\nu=s$ or $\nu=s+d$, then it is totally umbilical
\cite[Thm.~3.14]{CFG}.  Consequently, when $0<s<n$, a non-totally
umbilical example with finite contact number greater than two must have an
indefinite normal bundle.  In particular, a Lorentzian surface with finite contact number greater
than two cannot lie in a Lorentz--Minkowski space of index one.  This is why
the next construction uses $\mathbb E_4^8$.
\end{remark}

Chen and Li exhibited a Euclidean surface in $\mathbb E^8$ with contact
number six \cite[Example~6.8]{ChenLi}; Chen's later paper contains further
Riemannian high-contact families \cite{ChenHighContact}.  The next
construction shows that genuinely higher finite contact also occurs in
indefinite signature.

\begin{example}[A Lorentzian surface with contact number six]
\label{ex:lorentz-contact-six}
Put
\begin{equation*}
 a=\sqrt{1+\frac1{\sqrt2}},
 \qquad
 b=\sqrt{1-\frac1{\sqrt2}}.
\end{equation*}
Let $\mathbb E_1^2$ carry the metric $dx^2-dy^2$, and endow
$\mathbb E_4^8$ with
\begin{equation*}
 dz_1^2+dz_2^2-dz_3^2-dz_4^2
 +dz_5^2+dz_6^2-dz_7^2-dz_8^2.
\end{equation*}
For $c,d>0$, set
\begin{equation*}
\begin{aligned}
 \Theta_{c,d}(x,y)=\bigl(&\cos(cx)\cosh(dy),\ \sin(cx)\cosh(dy),\\
                         &\sin(cx)\sinh(dy),\ \cos(cx)\sinh(dy)\bigr).
\end{aligned}
\end{equation*}
Then
\begin{equation*}
 \Phi(x,y)=\frac1{\sqrt2}
 \bigl(\Theta_{a,b}(x,y),\Theta_{b,a}(x,y)\bigr)
\end{equation*}
defines an isometric immersion
$\Phi\colon\mathbb E_1^2\to\mathbb E_4^8$ whose ordinary, spacelike and
timelike contact numbers are all equal to six.

Indeed, let $G_0=\operatorname{diag}(1,1,-1,-1)$ and
$G=G_0\oplus G_0$, where $\oplus$ denotes the block-diagonal direct
sum of matrices.  For $c>0$, set
\begin{equation*}
 R_c=
 \begin{pmatrix}
 0&-c&0&0\\ c&0&0&0\\ 0&0&0&c\\ 0&0&-c&0
 \end{pmatrix},
 \qquad
 S_c=
 \begin{pmatrix}
 0&0&0&c\\ 0&0&c&0\\ 0&c&0&0\\ c&0&0&0
 \end{pmatrix}.
\end{equation*}
Define
\begin{equation*}
 P=R_a\oplus R_b,
 \qquad Q=S_b\oplus S_a,
 \qquad
 v_0=\frac1{\sqrt2}(1,0,0,0,1,0,0,0)^{\mathsf t}.
\end{equation*}
A direct multiplication gives
\begin{equation*}
 P^{\mathsf t}G+GP=0,
 \qquad Q^{\mathsf t}G+GQ=0,
 \qquad PQ=QP,
\end{equation*}
and $\Phi(x,y)=\exp(xP+yQ)v_0$.  Moreover, for every
$(x_0,y_0)\in\mathbb E_1^2$,
\begin{equation*}
 \Phi(x+x_0,y+y_0)=\exp(x_0P+y_0Q)\Phi(x,y).
\end{equation*}
Since $P$ and $Q$ are $G$-skew-adjoint and commute,
$\exp(x_0P+y_0Q)$ is an ambient linear isometry.  Hence all the
extrinsic data are transported equivariantly, and it suffices to work at
the origin.  There,
\begin{equation*}
 e_1:=\partial_x\Phi(0,0)=Pv_0,
 \qquad e_2:=\partial_y\Phi(0,0)=Qv_0.
\end{equation*}
A direct computation gives
$\langle e_1,e_1\rangle=1$,
$\langle e_2,e_2\rangle=-1$ and
$\langle e_1,e_2\rangle=0$.  By equivariance, the same identities
hold at every point, and hence $\Phi$ is isometric.

For $u=pe_1+qe_2$, put $L=pP+qQ$, and let
\begin{equation*}
 \pi^\perp z=z-\langle z,e_1\rangle e_1+\langle z,e_2\rangle e_2
\end{equation*}
be the normal projection at the origin.  Define normal vectors recursively
by
\begin{equation*}
 N_0=\pi^\perp(L^2v_0),
 \qquad
 N_{j+1}=\pi^\perp(LN_j).
\end{equation*}
Extend $u$ to the parallel field $p\,\partial_x+q\,\partial_y$ on
$\mathbb E_1^2$.  Since this field is parallel and the immersion is
equivariant, a normal field determined by a vector $\xi$ at the origin
is transported as $\exp(xP+yQ)\xi$.  Differentiation in the direction
$u$ therefore gives $L\xi$, and taking its normal component shows that
normal covariant differentiation at the origin is exactly
$\xi\mapsto\pi^\perp(L\xi)$.  It follows inductively that
\begin{equation*}
 N_j=(\overline\nabla^{\,j}\II)(u^{j+2}),
 \qquad j\geq0.
\end{equation*}
Since the induced metric is constant in the coordinates $(x,y)$, its
Levi--Civita connection vanishes on the coordinate vector fields.  Hence, by
the Gauss formula,
\begin{equation*}
 LPv_0=\II(u,e_1),
 \qquad
 LQv_0=\II(u,e_2),
\end{equation*}
and both vectors are normal at the origin.  Write
\begin{equation*}
 A_{N_j}u=C_1^{(j)}e_1+C_2^{(j)}e_2.
\end{equation*}
Then
\begin{equation*}
 C_1^{(j)}=\langle LPv_0,N_j\rangle,
 \qquad
 C_2^{(j)}=-\langle LQv_0,N_j\rangle.
\end{equation*}
Using $a^2+b^2=2$, $a^2-b^2=\sqrt2$ and $ab=1/\sqrt2$, the above recursion gives
\begin{align*}
 A_{N_0}u&=\frac32(p^2-q^2)u,
 &A_{N_1}u&=0,\\
 A_{N_2}u&=-\frac14(p^2-q^2)^2u,
 &A_{N_3}u&=0.
\end{align*}
Thus every non-null $u$ is an eigenvector of the shape operators
corresponding to $j=0,1,2,3$.  The equivariance of the construction
transports these identities to every point.  Hence
Theorem~\ref{thm:contact-criterion} gives contact of order at least six.
At the next order, the obstruction to proportionality is
\begin{equation*}
 pC_2^{(4)}-qC_1^{(4)}
 =pq(p^2+q^2)(p^4+6p^2q^2+q^4).
\end{equation*}
This is the determinant of the components of $u$ and $A_{N_4}u$ in the
basis $e_1,e_2$.  For the unit spacelike vector
$u=\sqrt2e_1+e_2$, its value is $51\sqrt2\neq0$.  Hence $u$ is not an
eigenvector of $A_{N_4}$, and spacelike contact of order seven fails.
Therefore $\cplus(\mathbb E_1^2)=6$, and
Theorem~\ref{thm:causal-invariance} yields
\begin{equation*}
 \cnum(\mathbb E_1^2)=\cplus(\mathbb E_1^2)
 =\cminus(\mathbb E_1^2)=6.
\end{equation*}
\end{example}

\subsection{Geodesic normal sections and helical immersions}

We say that the immersion has \emph{geodesic normal sections} if
$\beta_u=\gamma_u$ locally for every unit non-null tangent vector $u$.
Such an immersion has contact of every order.  The converse also holds,
without any analyticity assumption.

\begin{proposition}[The infinite-contact limit]
The following assertions are equivalent:
\begin{enumerate}
 \item the immersion has geodesic normal sections;
 \item $\cnum(M_s^n)=\infty$;
 \item $\cplus(M_s^n)=\infty$;
 \item $\cminus(M_s^n)=\infty$.
\end{enumerate}
\end{proposition}

\begin{proof}
The first assertion implies the second by definition, and
Theorem~\ref{thm:causal-invariance} gives the equivalence of the last
three.  Conversely, assume that one of the two one-sided contact numbers
is infinite, and denote the corresponding causal sign by
$\varepsilon\in\{1,-1\}$.  Given $j\geq0$ and
$u\in U_\varepsilon M_s^n$, apply
Theorem~\ref{thm:contact-criterion} with $k=j+3$.  Then $u$ is an
eigenvector of $A_{(\overline\nabla^{\,j}\II)(u^{j+2})}$, say
\begin{equation*}
 A_{(\overline\nabla^{\,j}\II)(u^{j+2})}u=\alpha_j u.
\end{equation*}
Taking the scalar product with $u$ and using \eqref{eq:shape-II} gives
\begin{equation*}
 \alpha_j=\varepsilon
 \langle(\overline\nabla^{\,j}\II)(u^{j+2}),\II(u,u)\rangle.
\end{equation*}
Thus the condition in \cite[Lemma~3.3(iii)]{MiuraGNS} holds for every
$j\geq0$, and the immersion has geodesic normal sections.
\end{proof}

The semi-Riemannian converse from geodesic normal sections to helicality
requires extra hypotheses.  Miura proved that spacelike and timelike
helicality are equivalent and that helical immersions in semi-Riemannian
space forms have geodesic normal sections \cite{MiuraHelical}.  He later
showed that the converse is recovered by adding a geodesic
non-degeneracy order and the proper-order condition $(P_d)$
\cite[Cor.~3.10]{MiuraGNS}.  Thus the literal indefinite extension of the
Riemannian equivalence recalled in \cite[Thm.~4.41]{GomezThesis} is false:
osculating spaces may be degenerate, and the proper order of the image of a
geodesic may vary with its initial direction.

\begin{example}[Expansions]\label{ex:expansions}
Let $F=(f_1,\ldots,f_\ell)\colon\mathbb E_s^n\to\mathbb R^\ell$ be smooth,
and consider
\begin{equation*}
 \Psi_F(x)=(F(x),x,F(x))
 \colon\mathbb E_s^n\longrightarrow\mathbb E_{s+\ell}^{n+2\ell},
\end{equation*}
where the first $\ell$ coordinates are negative definite and the last
$\ell$ coordinates are positive definite.  This general expansion appears
in \cite[Example~2.1]{MiuraGNS}; the case $\ell=1$ is
\cite[Example~3.2]{CFG}.

Let $e_1,\ldots,e_\ell$ be the standard basis of $\mathbb R^\ell$ and put
$\xi_\alpha=(e_\alpha,0,e_\alpha)$.  The space
$\mathcal N=\Span\{\xi_1,\ldots,\xi_\ell\}$ is a constant totally null
normal subspace, and
\begin{equation}\label{eq:expansion-II}
 \II(X,Y)=\sum_{\alpha=1}^\ell
 \operatorname{Hess}f_\alpha(X,Y)\,\xi_\alpha.
\end{equation}
Every iterated covariant derivative of $\II$ takes values in $\mathcal N$.
Since $\mathcal N$ is totally null, \eqref{eq:shape-II} and
\eqref{eq:expansion-II} imply that all the corresponding shape operators
vanish.  Hence the ordinary, spacelike and timelike contact numbers are
all infinite.

The geodesic-normal-section property is also direct.  If
$\gamma(t)=p+tu$ is a geodesic of $\mathbb E_s^n$, then
\begin{equation*}
\begin{aligned}
 &\Psi_F(p+tu)-\Psi_F(p)-t\,d\Psi_{F,p}(u)\\
 &\qquad=\bigl(F(p+tu)-F(p)-t\,dF_p(u),0,
 F(p+tu)-F(p)-t\,dF_p(u)\bigr)\in\mathcal N.
\end{aligned}
\end{equation*}
Thus $\Psi_F\circ\gamma$ lies in the affine normal-section space.

These immersions need not be helical.  Take $n=2$, $s=1$, $\ell=1$, let
$\mathbb E_1^2$ carry the metric $dx^2-dy^2$, and set $f(x,y)=x^3$.
At the origin the timelike unit direction $e_2$ is mapped
to a straight line, whereas the timelike unit direction
$u=(e_1+2e_2)/\sqrt3$ is mapped to
\begin{equation*}
 t\longmapsto
 \left(\frac{t^3}{3\sqrt3},\frac{t}{\sqrt3},
       \frac{2t}{\sqrt3},\frac{t^3}{3\sqrt3}\right),
\end{equation*}
which is not a line but is contained in a two-dimensional affine plane.
The two geodesics have different proper orders.  Hence the immersion has
geodesic normal sections and infinite contact but is not helical.
\end{example}

\begin{remark}
Examples~\ref{ex:cylinder}, \ref{ex:lorentz-contact-six}, and~\ref{ex:expansions} exhibit, in indefinite signature, contact
numbers $2$, $6$, and $\infty$, respectively.  The last example also shows
why the extra non-degeneracy hypotheses in the indefinite helical theory
are essential.
\end{remark}


\begin{thebibliography}{99}

\bibitem{Besse}
A.~L. Besse,
\emph{Manifolds All of Whose Geodesics Are Closed},
Ergebnisse der Mathematik und ihrer Grenzgebiete, vol.~93,
Springer-Verlag, Berlin--Heidelberg--New York, 1978.
\url{https://doi.org/10.1007/978-3-642-61876-5}

\bibitem{BoumukiMaeda}
N.~Boumuki and S.~Maeda,
\emph{Study of isotropic immersions},
Kyungpook Math. J. \textbf{45} (2005), no.~3, 363--394.

\bibitem{CFG}
J.~L. Cabrerizo, M.~Fern\'andez and J.~S. G\'omez,
\emph{The contact number of a pseudo-Euclidean submanifold},
Taiwanese J. Math. \textbf{12} (2008), no.~7, 1707--1720.
\url{https://doi.org/10.11650/twjm/1500405081}

\bibitem{CFGRigidity}
J.~L. Cabrerizo, M.~Fern\'andez and J.~S. G\'omez,
\emph{Rigidity of pseudo-isotropic immersions},
J. Geom. Phys. \textbf{59} (2009), no.~7, 834--842.
\url{https://doi.org/10.1016/j.geomphys.2009.03.006}

\bibitem{CFGMarginallyTrapped}
J.~L. Cabrerizo, M.~Fern\'andez and J.~S. G\'omez,
\emph{Isotropy and marginally trapped surfaces in a spacetime},
Class. Quantum Grav. \textbf{27} (2010), no.~13, 135005, 12 pp.
\url{https://doi.org/10.1088/0264-9381/27/13/135005}

\bibitem{CFGIsotropicSubmanifolds}
J.~L. Cabrerizo, M.~Fern\'andez and J.~S. G\'omez,
\emph{Isotropic submanifolds of pseudo-Riemannian spaces},
J. Geom. Phys. \textbf{62} (2012), no.~9, 1915--1924.
\url{https://doi.org/10.1016/j.geomphys.2012.05.002}

\bibitem{ChenHighContact}
B.-Y. Chen,
\emph{Surfaces with high contact number and their characterization},
Ann. Mat. Pura Appl. (4) \textbf{185} (2006), no.~4, 555--576.
\url{https://doi.org/10.1007/s10231-005-0169-1}

\bibitem{ChenLi}
B.-Y. Chen and S.-J. Li,
\emph{The contact number of a Euclidean submanifold},
Proc. Edinb. Math. Soc. (2) \textbf{47} (2004), no.~1, 69--100.
\url{https://doi.org/10.1017/S0013091503000038}

\bibitem{ChenVerheyen}
B.-Y. Chen and P.~Verheyen,
\emph{Submanifolds with geodesic normal sections},
Math. Ann. \textbf{269} (1984), no.~3, 417--429.
\url{https://doi.org/10.1007/BF01450703}

\bibitem{Fueki}
S.~Fueki,
\emph{Pointed helical submanifolds of pseudo-Riemannian manifolds},
J. Geom. \textbf{62} (1998), 129--143.
\url{https://doi.org/10.1007/BF01237605}

\bibitem{HongHouh}
Y.~Hong and C.-S.~Houh,
\emph{Helical immersions and normal sections},
Kodai Math. J. \textbf{8} (1985), no.~2, 171--192.
\url{https://doi.org/10.2996/kmj/1138037046}

\bibitem{GomezThesis}
J.~S. G\'omez Casanueva,
\emph{Inmersiones isotr\'opicas pseudo-riemannianas},
Ph.D. thesis, Universidad de Sevilla, 2008. Available at
\href{https://idus.us.es/items/2eb9696c-419f-4523-8788-6b90fd29ab01}{idUS}.

\bibitem{ItohOgiue}
T.~Itoh and K.~Ogiue,
\emph{Isotropic immersions},
J. Differential Geom. \textbf{8} (1973), no.~2, 305--316.

\bibitem{MaedaTsukada}
S.~Maeda and K.~Tsukada,
\emph{Isotropic immersions into a real space form},
Canad. Math. Bull. \textbf{37} (1994), no.~2, 245--253.
\url{https://doi.org/10.4153/CMB-1994-036-5}

\bibitem{MiuraHelical}
K.~Miura,
\emph{Helical geodesic immersions of semi-Riemannian manifolds},
Kodai Math. J. \textbf{30} (2007), no.~3, 322--343.
\url{https://doi.org/10.2996/kmj/1193924937}

\bibitem{MiuraGNS}
K.~Miura,
\emph{Isometric immersions with geodesic normal sections in
semi-Riemannian geometry},
Tokyo J. Math. \textbf{31} (2008), no.~2, 479--488.
\url{https://doi.org/10.3836/tjm/1233844064}

\bibitem{Miura}
K.~Miura,
\emph{The contact number of an affine immersion and its upper bounds},
Results Math. \textbf{56} (2009), 245--258.
\url{https://doi.org/10.1007/s00025-009-0444-3}

\bibitem{ONeillIsotropic}
B.~O'Neill,
\emph{Isotropic and K\"ahler immersions},
Canad. J. Math. \textbf{17} (1965), 907--915.
\url{https://doi.org/10.4153/CJM-1965-086-7}

\bibitem{ONeill}
B.~O'Neill,
\emph{Semi-Riemannian Geometry with Applications to Relativity},
Pure and Applied Mathematics, vol.~103,
Academic Press, New York, 1983.

\bibitem{Verheyen}
P.~Verheyen,
\emph{Submanifolds with geodesic normal sections are helical},
Rend. Sem. Mat. Univ. Politec. Torino \textbf{43} (1985), 511--527.

\end{thebibliography}
\end{document}